\documentclass[12pt,leqno,fleqn]{amsart}
\usepackage[english]{babel}
\usepackage{amsmath, amsthm, amsfonts, mathrsfs, amssymb, float}
\usepackage{mathtools}
\mathtoolsset{centercolon}
\usepackage{booktabs}
\usepackage[shortlabels]{enumitem}
\setlist[itemize]{leftmargin=20pt}
\usepackage[colorlinks]{hyperref}
\usepackage{tikz}
\usetikzlibrary{arrows}
\usepackage{pifont}
\usepackage{color}
\usepackage{graphicx}

\mathtoolsset{showonlyrefs,showmanualtags}
\usepackage[msc-links]{amsrefs}

\newtheorem{theorem}{Theorem}

\newtheorem{lemma}[theorem]{Lemma}

\newtheorem{remark}[theorem]{Remark}

\theoremstyle{definition}
\newtheorem{definition}[theorem]{Definition}

\numberwithin{theorem}{section}
\numberwithin{equation}{section}

\title[$A_\infty$ Weights in Martingale Spaces]{Characterizations of $A_\infty$ Weights in Martingale Spaces}

\author[J. Ju]{Jie Ju}
\address{ School of Mathematical Sciences, Yangzhou University, Yangzhou 225002, China}
\email {jieju$\_$yzu@163.com}

\author[W. Chen]{Wei Chen}
\address{ School of Mathematical Sciences, Yangzhou University, Yangzhou 225002, China}
\email {weichen@yzu.edu.cn}

\author[J. Y. Cui]{Jingya Cui}
\address{ School of Mathematical Sciences, Yangzhou University, Yangzhou 225002, China}
\email {jycui$\_$yzu@163.com}

\author[C. Zhang]{Chao Zhang}
\address{School of Statistics and Mathematics, Zhejiang Gongshang University, Hangzhou 310018, China}
\email{zaoyangzhangchao@163.com}

\thanks{W. Chen is supported by the National Natural Science Foundation of China (11971419, 12271469). J. Ju is supported by the Jiangsu Students' Platform for Innovation and Entrepreneurship Training Program (202211117018Z). C. Zhang is supported by the National Natural Science Foundation of China (11971431) and the Natural Science Foundation of Zhejiang Province (LY22A010011).}

\allowdisplaybreaks
\begin{document}
  \begin{abstract}
Grafakos systematically
proved that $A_\infty$ weights have different characterizations for cubes in Euclidean spaces in his classical text book. Very recently, Duoandikoetxea, Mart\'{\i}n-Reyes, Ombrosi and Kosz discussed several characterizations of the $A_{\infty}$ weights in the setting of general bases. By conditional expectations, we study $A_\infty$ weights in martingale spaces. Because conditional expectations are Radon-Nikod\'{y}m derivatives with respect to sub$\hbox{-}\sigma\hbox{-}$fields which have no geometric structures, we need new ingredients. Under a regularity assumption on weights, we obtain  equivalent characterizations of the $A_{\infty}$ weights. Moreover, using weights modulo conditional expectations, we have one-way implications of different characterizations.
\end{abstract}

\keywords{weight, conditional expectation, maximal operator, median}

\subjclass[2010]{Primary: 60G46; Secondary: 60G48, 60G42}

\maketitle

\section{Introduction }
As is well known, a non-negative function $\omega$ on $\mathbb{R}^n$ is an $A_p$ weight with $p>1,$ if there exists a constant $C$ for all cubes $Q$ such that
$$
\bigg( \frac{1}{|Q|} \int_Q \omega \, d\mu \bigg) \bigg( \frac{1}{|Q|} \int_Q \omega^{-\frac{1}{p-1}} \, d\mu  \bigg)^{p-1} \leq C.
$$
Muckenhoupt \cite{MR293384} observed that the weight has an open property $A_p=\cup_{1<q<p}A_q.$ Moreover, Muckenhoupt \cite{MR350297} defined the $A^M_\infty$ weight as follows: there exist $0<\varepsilon,~\delta<1$ such that for all $E\subseteq Q$ it holds that
$$
|E| < \delta |Q| \Rightarrow \omega(E) < \varepsilon \omega(Q).
$$
Then he showed $A^M_{\infty}=\cup_{p>1}A_p.$ Independently,
Coifman and Fefferman \cite{MR358205} introduced an $A^{CF}_\infty$ weight and proved $A^{CF}_{\infty}=\cup_{p>1}A_p,$ where the $A^{CF}_\infty$ weight $\omega$ is defined as follows: there exist $C,~\delta>0$ such that for all $E\subseteq Q$
$$
\frac{\omega(E)}{\omega(Q)} \leq C \bigg( \frac{|E|}{|Q|} \bigg)^\delta.
$$
Later, a condition $A_{\infty}^{exp}$ defined by a limit of the $A_p$ weight as $p\uparrow\infty$ was studied almost simultaneously in \cite{MR727244} and \cite[p.405]{MR807149} and $\cup_{p>1}A_p=A^{exp}_{\infty}.$

As we have seen, $A^{M}_\infty,$ $A^{CF}_\infty$ and $A_{\infty}^{exp}$ are equivalent. These are geometric characterizations of $\cup_{p>1}A_p$ and systematically studied in Grafakos \cite[Theorem 7.3.3]{MR3243734}. Very recently, Duoandikoetxea, Mart\'{\i}n-Reyes and Ombrosi \cite{MR3473651} compared and discussed different characterizations of $\cup_{p>1}A_p$ in the setting of general bases. Indeed, they studied many other characterizations which are not geometric. Here we list four  characterizations for cubes mentioned in \cite{MR3473651}:
\begin{enumerate}
  \item [($A^{*}_{\infty}$)]There exists $C > 0$ such that
	$$\int_Q M (\omega \chi_Q) dx \leq C \omega(Q),$$
where $\omega(Q) := \int_Q \omega  d\mu$ and $M(\cdot)$ is the Hardy-Littlewood maximal operator (see \cite{MR481968}, \cite{MR883661}). Hyt\"{o}nen and P\'{e}rez \cite{MR3092729} used the weight $A^{*}_{\infty}$ to  improve estimates of the bounds in the weighted inequalities.

  \item [($A^{log}_{\infty}$)] There exists $C > 0$ such that
	$$\int_Q \omega \log^+ \frac{\omega}{\omega_Q}dx \leq C \omega(Q),$$
where $\omega_Q:=\omega(Q)/|Q|$ (see \cite{MR481968}).

  \item [($A^{med}_{\infty}$)] There exists $C > 0$ such that
	$$\omega_Q \leq C m(\omega, Q)$$
where the median of $\omega$ in $Q$ is a number $m(\omega,Q)$ such that $|\{x \in Q : \omega(x) < m(\omega,Q)\}| \leq |Q| / 2$ and $|\{x \in Q : \omega(x) > m(\omega,Q)\}| \leq |Q| / 2$ (see \cite{MR529683}). Using the median, Lerner \cite{MR2721744}
obtain a decomposition of an arbitrary measurable function in terms of local mean oscillations.

\item [($A^{\lambda}_{\infty}$)] There exist $C, \beta > 0$ such that
$$\omega\big(\{  x \in Q : \omega(x) > \lambda \}\big) \leq C \lambda \, \big|\{x \in Q : \omega(x) > \beta \lambda \}\big|,$$
where $\lambda>\omega_Q$. This kind of characterization appeared independently in \cite{MR402038} and \cite{MR358205}.
\end{enumerate}

Although $A^{*}_{\infty},$ $A^{log}_{\infty},$ $A^{med}_{\infty}$ and $A^{\lambda}_{\infty}$ are not geometric, they are equivalent to $A^{M}_\infty,$ $A^{CF}_\infty$ and $A_{\infty}^{exp}$ for cubes. In
the context of general bases,
the relations between them are more complicated (see \cite{MR3473651} and \cite{MR4446233} for more information).

In this paper, we study $A_\infty$ weights in martingale spaces. Izumisawa and Kazamaki \cite{MR436313} first introduced the
$A_p$ weight for martingales. Under some additional
conditions, they obtained the open property $A_p=\cup_{1<q<p}A_q.$ In martingale setting, it is well known that this property is false in general, because Bonami and L\'{e}pingle \cite{MR544802} showed that for any $p>1,$ there exists a weight $\omega\in A_p,$ but $\omega\notin A_{p-\varepsilon}$ for all $\varepsilon>0$ (see also \cite[p.241]{MR1224450}).
Under some additional restrictions, the $A_p$ weight was extensively studied by Dol\'{e}ans-Dade and Meyer \cite{MR544804}. Motivated by the work
of \cite[Theorem 7.3.3]{MR3243734}, \cite{MR3473651} and \cite{MR4446233}, we study
several characterizations of $A_{\infty}$ weights in the setting of martingales. Our first result is the following Theorem \ref{Thm:equa},  which partially depends on a regularity condition $S$ (see Definition \ref{regular_p}).

\begin{theorem}\label{Thm:equa}Let $\omega$ be a weight. If $\omega\in S,$ then the following are
equivalent.

\begin{enumerate}[\rm (1)]

\item \label{Thm:equa_Ap}There exist
$C,~p>1$ such that for all $n\in \mathbb{N}$ we have
\begin{equation}\label{Ap}
\mathbb{E}(\omega|\mathcal {F}_n)\mathbb{E}(\omega^{-\frac{1}{p-1}}|\mathcal
{F}_n)^{p-1}\leq C,
\end{equation}
which is denoted by $\omega\in\bigcup\limits_{p>1}A_p.$

\item \label{Thm:equa_A_exp_infty}There exists a positive
constant $C$ such that for all $n\in
\mathbb{N}$ we have
\begin{equation}\label{A_exp_infty}
\mathbb{E}(\omega|\mathcal {F}_n)\leq C\exp \mathbb{E}(\log\omega|\mathcal {F}_n),
\end{equation}
which is denoted by $\omega\in A^{exp}_{\infty}.$

\item \label{Thm:equa_wAinfty}There exist $0<\gamma,~\delta<1$ such that for all $n\in
\mathbb{N}$ we have
\begin{equation}\label{wAinfty}
\mathbb{E}(\chi_{\{\omega\leq\gamma\omega_n\}}|\mathcal
{F}_n)\leq\delta<1,
\end{equation}
which is denoted by $\omega\in A^{con}_{\infty}.$

\item \label{Thm:equa_R}There exist $0<\alpha,~\beta<1$ such that for all $n\in
\mathbb{N}$ and $A\in\mathcal {F}$ we have
\begin{equation}
\label{R}\mathbb{E}(\chi_A|\mathcal {F}_n)\leq\alpha<1\Rightarrow
\mathbb{E}_\omega(\chi_A|\mathcal {F}_n)\leq\beta<1,\end{equation}
which is denoted by $\omega \in A_{\infty}^{M}.$

\item \label{Thm:equa_RH} There exist $C,~q>1$ such that for all $n\in \mathbb{N}$ we have
\begin{equation}\label{RH}
\mathbb{E}(\omega^q|\mathcal {F}_n)\leq
C\mathbb{E}(\omega|\mathcal {F}_n)^q,\end{equation}
which is the reverse H\"{o}lder condition and denoted by $\omega\in \bigcup\limits_{q>1}RH_{q}$.

\item \label{Thm:equa_RR}There exist $0<\varepsilon'<1$ and $C>1$ such that
for all $n\in \mathbb{N}$ we have
\begin{equation}
\label{RR}\mathbb{E}_\omega(\chi_A|\mathcal {F}_n)\leq C\mathbb{E}(\chi_A|\mathcal
{F}_n)^{\varepsilon'}.
\end{equation}
which is denoted by $\omega\in A^{CF}_{\infty}.$

\item \label{Thm:equa_rR} There exist $0<\alpha,~\beta<1$ such that for all $n\in
\mathbb{N}$ and $A\in\mathcal {F}$ we have
\begin{equation}
\label{rR}\mathbb{E}_\omega(\chi_A|\mathcal {F}_n)\leq\alpha<1\Rightarrow
\mathbb{E}(\chi_A|\mathcal {F}_n)\leq\beta<1,\end{equation}
which is denoted by $\omega \in \hat{A}_{\infty}^{M}.$

\end{enumerate}
\end{theorem}

\begin{remark}In the setting of martingale spaces, Long \cite[Definitions 6.1.6 and 7.1.1]{MR1224450} defined $A_\infty$ weights and $R$ conditions by \eqref{Ap} and the right hand side of \eqref{regular_con}, respectively. If $\omega\in R,$ \cite[Propositions 6.3.8 and 6.3.9]{MR1224450} showed $\ref{Thm:equa_Ap}\Rightarrow\ref{Thm:equa_RH}$ which is one-way implication. In our paper, we assume $\omega\in S$ in place of $\omega\in R.$
\end{remark}

We mention that $\omega\in S$ is used only in $\ref{Thm:equa_R}\xRightarrow{S}\ref{Thm:equa_RH}$ and
$\ref{Thm:equa_rR}\xRightarrow{S}\ref{Thm:equa_Ap}$ in Theorem \ref{Thm:equa}.
It is natural to discuss what happens without the condition $S.$ Using weights modulo
conditional expectations $\omega/\omega_n$ instead of weights $\omega,$ we study several characterizations of $A_\infty$ weights.
First we have Theorems \ref{thm:Reverse-test} and \ref{thm:Cond}.

\begin{theorem}\label{thm:Reverse-test}The following statements are equivalent.
\begin{enumerate}
\item \label{Reverse-test1}$\omega\in \bigcup\limits_{q>1}RH_{q}$.

\item \label{Reverse-test2}$\omega\in A^{CF}_{\infty}.$
\end{enumerate}
\end{theorem}

\begin{theorem}\label{thm:Cond}The following statements are equivalent.
\begin{enumerate}
\item \label{thm:Cond1}$\omega\in A^{con}_{\infty}.$
\item \label{thm:Cond2}$\omega \in A_{\infty}^{M}.$
\end{enumerate}
\end{theorem}

Using a kind of reverse H\"{o}lder condition which appeared in Str\"{o}mberg and Wheeden \cite{MR766221}, we give a characterization of $\omega\in A^{exp}_{\infty},$ which is Theorem \ref{thm:exp-s}.
\begin{theorem}\label{thm:exp-s}The following statements are equivalent.
\begin{enumerate}
\item \label{thm:exp-s2} $\omega\in A^{exp}_{\infty}.$
\item \label{thm:exp-s1}There exists $C>1$ such that for every $s\in(0,1)$ we have
\begin{equation}\label{thm:eq-exp-s}
\mathbb{E}(\omega|\mathcal {F}_n)\leq C\mathbb{E}(\omega^s|\mathcal {F}_n)^{\frac{1}{s}},\end{equation}
which is denoted by $\omega\in A^{SW}_{\infty}.$
\end{enumerate}
\end{theorem}

Modifying $A_{\infty}^{\lambda}$ and $A_{\infty}^{med}$ in
\cite{MR402038} and \cite{MR529683}, respectively, we have
one-way implications. Indeed,
introducing the quotient $\omega/\omega_n$ into $A_{\infty}^{\lambda}$ of \cite{MR402038}, we obtain Theorem \ref{thm:imp1}.

\begin{theorem}\label{thm:imp1}Let $\omega$ be a weight. We have the sequence of implications
$\eqref{thm:lev}\Rightarrow\eqref{thm:rev} \Rightarrow\eqref{thm:log}\Rightarrow\eqref{thm:imp_Cond2}$
for the following statements.
\begin{enumerate}
\item \label{thm:lev}There exist $0<\beta<1$ and $C>1$ such that for all $n\in \mathbb{N}$ and $\lambda>1$ we have
\begin{equation}\label{thm:equ_lev}
\mathbb{E}_\omega(\chi_{\{\frac{\omega}{\omega_n}>\lambda\}}|\mathcal {F}_n)\leq C\lambda\mathbb{E}(\chi_{\{\frac{\omega}{\omega_n}>\beta\lambda\}}|\mathcal {F}_n),\end{equation}
which is denoted by $\omega\in A_{\infty}^{\lambda}.$
\item \label{thm:rev}
$\omega\in\bigcup\limits_{q>1}RH_q.$
\item \label{thm:log}There exists $C>1$ such that for all $n\in \mathbb{N}$ we have
\begin{equation}\label{thm:equ_log}
\mathbb{E}_{\omega}(\log^+\frac{\omega}{\omega_n}|\mathcal {F}_n)\leq C,
\end{equation}
which is denoted by $\omega\in A_{\infty}^{log}.$
\item \label{thm:imp_Cond2}$\omega \in A_{\infty}^{M}.$
\end{enumerate}
\end{theorem}

As for $A_{\infty}^{med}$ in \cite{MR529683}, we replace the median $m(\omega,Q)$ by the median function $m(\omega,n)$ (see Definition \ref{media_f}), which is the key observation in Theorem \ref{thm:imp2}.

\begin{theorem}\label{thm:imp2}Let $\omega$ be a weight. We have the sequence of implications
$\eqref{thm:imp2_Ap}\Rightarrow\eqref{thm:imp2_Log} \Rightarrow\eqref{thm:imp2_Mid}\Rightarrow\eqref{thm:imp2_Dou}$
for the following statements.
\begin{enumerate}
\item \label{thm:imp2_Ap}$\omega\in\bigcup\limits_{p>1} A_p.$

\item \label{thm:imp2_Log}$\omega\in A^{exp}_{\infty}.$

\item \label{thm:imp2_Mid} There exists $C>1$ such that for all $n\in \mathbb{N}$ we have
\begin{equation}\label{thm:imp2_EMid}
\omega_n\leq Cm(\omega,n),
\end{equation}
which is denoted by $\omega\in A_{\infty}^{med}.$

\item \label{thm:imp2_Dou}$\omega \in A_{\infty}^{M}.$
\end{enumerate}
\end{theorem}

Now we give Theorem \ref{thm:Wi} which is related to $A^*_{\infty}$ in \cite{MR481968} and \cite{MR883661}.
The main ingredient of Theorem \ref{thm:Wi} is the conditional expectation of tailed maximal operators (see Definition \ref{tailed_o}). It is worth observing that \eqref{tailed_mo} equals $\mathbb{E}(M^*_n(\omega/\omega_n)|\mathcal {F}_n)\leq C.$
The tailed maximal operators first appeared in \cite{MR1301765} and were used to proved two-weight inequalities for martingales under some additional assumption.
In view of Theorem \ref{thm:Wi}, we have $A^{exp}_{\infty}\subseteq A^*_{\infty}.$

\begin{theorem}\label{thm:Wi} Given the following statements.
\begin{enumerate}
\item \label{thm:Wi_Ap}$\omega\in \bigcup\limits_{p>1}A_p.$

\item \label{thm:Wi_rev}
$\omega\in\bigcup\limits_{q>1}RH_q.$

\item \label{thm:Wi_log}$\omega\in A^{log}_{\infty}.$

\item \label{Thm:Wi_exp}$\omega\in A^{exp}_{\infty}.$
\end{enumerate}
Then each of these statements implies $\omega\in A^*_{\infty},$ i.e., for all $n\in\mathbb{N},$
\begin{equation}\label{tailed_mo}
\mathbb{E}(M^*_n(\omega)|\mathcal {F}_n)\leq C\omega_n.
\end{equation}
\end{theorem}

The paper is organized as follows. Some preliminaries are contained in Sect. \ref{preli}.
In Sect. \ref{regular} we prove Theorem \ref{Thm:equa} for regular weights. Sect. \ref{without} is devoted to
theorems without additional assumptions.

\section{Preliminaries}\label{preli}

Let $(\Omega,\mathcal {F},\mu)$ be a complete probability space and
$(\mathcal {F}_n)_{n\geq0}$ an increasing sequence of
sub$\hbox{-}\sigma\hbox{-}$fields of $\mathcal{F}$ with
$\mathcal{F}=\bigvee_{n\geq0}\mathcal{F}_n.$
The conditional expectation with respect
 to $(\Omega,\mathcal{F},\mu,\mathcal{F}_n)$ is denoted by $\mathbb{E}(\cdot|\mathcal{F}_n).$
In this paper, for $p\geq1,$ a martingale
$f=(f_n)_{n\geq0}\in L^p(\omega)$ is meant as $f_n=E(f|\mathcal
{F}_n),~f\in L^p(\omega).$ A weight
$\omega$ is a random variable with $\omega>0$ and $
E(\omega)<\infty.$ Without loss of generality, we may assume
$E(\omega)=1$ since otherwise we can replace $\omega$ by $\omega/E(\omega).$

\begin{definition}\label{tailed_o}The Doob maximal operator $M$ and the tailed maximal
operator $M_n^*$ for martingale $f=(f_n)$ are defined by
\begin{equation*}Mf=\sup\limits_{n\geq 0}|f_n|\hbox{ and }M_n^*f=\sup\limits_{m\geq n}|f_n|,
\end{equation*}
respectively.
\end{definition}

\begin{definition}\label{regular_p}The weight $\omega$ is said to satisfy the regularity condition $S$, if there exists $C>1$
such that for all $n\in \mathbb{N}$ we have
\begin{equation}\label{regular_con}\frac{1}{C}\omega_{n-1}\leq\omega_n\leq C\omega_{n-1},\end{equation}
which is denoted by $\omega\in S.$
\end{definition}

\begin{remark} The right hand side of $S$
trivially holds if $(\mathcal {F}_n)_{n\geq0}$ is regular. The regularity of $(\mathcal {F}_n)_{n\geq0}$ is used several times
in the literature and is a very important definition (see \cite{MR1224450} and \cite{MR1320508}). More exactly, both sides of the condition $S$ follow from the
regularity of $(\mathcal {F}_n)_{n\geq0}$ (see \cite{MR3895312}).
\end{remark}

Let $\omega$ be a weight. We denote the conditional expectation
with respect
to $(\Omega,\mathcal{F},\omega d\mu,\mathcal{F}_n)$ by $\mathbb{E}_{\omega}(\cdot|\mathcal{F}_n).$ It follows that
$$\mathbb{E}_{\omega}(\cdot|\mathcal{F}_n)=\mathbb{E}(\cdot\omega|\mathcal{F}_n)/\omega_n=\mathbb{E}(\cdot\frac{\omega}{\omega_n}|\mathcal{F}_n).$$ If $A\in \mathcal{F}$,
 we denote  $\int_A\omega d\mu $ by $|A|_\omega$ and $\int_Ad\mu$ by
$|A|$, respectively.
For
$(\Omega,\mathcal {F},\mu)$ and $(\mathcal {F}_n)_{n\geq0},$ the
family of all stopping times is denoted by $\mathcal {T}.$

\begin{definition}\label{media_f}The median function of $\omega$ relative to $\mathcal {F}_n$ is defined as a $\mathcal {F}_n$ measurable function $m(\omega,n)$ such that $\mathbb{E}(\chi_{\{\omega> m(\omega,n)\}}|\mathcal {F}_n)\leq 1/ 2$
and $\mathbb{E}(\chi_{\{\omega< m(\omega,n)\}}|\mathcal {F}_n)\leq  1/ 2.$

\end{definition}

We denote the set of non-negative integers by $\mathbb{N}$ and all integers
by $\mathbb{Z},$ respectively. Throughout the paper letter $C$ always denotes a positive constant which may be different in each occurrence.

\section{Equivalent Characterizations of Regular $A_{\infty}$ Weights}\label{regular}
The following Lemma \ref{key_lemma_sta} will be used in the proof of Theorem \ref{Thm:equa}.

\begin{lemma}\label{key_lemma_sta}Let $v$ be a positive measurable function
and let $0<s_0<+\infty$. If $v\in L^{s_0},$ then
\begin{equation}\label{key_lemma}
\mathbb{E}(v^s|\mathcal {F}_n)^{\frac{1}{s}}\downarrow\exp \mathbb{E}(\log v|\mathcal {F}_n),~\text{~as~}s\downarrow0^+,
\end{equation}\end{lemma}

\begin{proof}[Proof of Lemma \ref{key_lemma_sta}] Because of $v\in L^{s_0},$ we have $v\in L^{s}$ with $0<s<s_0.$
H\"{o}lder's inequality for the conditional expectation (\cite[p.3]{MR1224450}) gives
$\mathbb{E}(v^s|\mathcal {F}_n)^{\frac{1}{s}}\leq\mathbb{E}(v^t|\mathcal {F}_n)^{\frac{1}{t}}$ with $0<s<t<s_0.$

Following Jensen's inequality for the conditional expectation (\cite[p.5]{MR1224450}), we have
$$
\exp \mathbb{E}(\log v^s|\mathcal {F}_n)\leq\mathbb{E}(v^s|\mathcal {F}_n),
$$
which implies \begin{equation}\label{eq_left}\exp \mathbb{E}(\log v|\mathcal {F}_n)\leq\mathbb{E}(v^s|\mathcal {F}_n)^{\frac{1}{s}}.
\end{equation}
Because of $x\leq \exp(x-1)$ for $x>0,$ then
$$\mathbb{E}(v^s|\mathcal {F}_n)\leq\exp \big(\mathbb{E}(v^s|\mathcal {F}_n)-1\big).$$
It follows that
\begin{equation}\label{eq_right}\mathbb{E}(v^s|\mathcal {F}_n)^{\frac{1}{s}}\leq\exp (\frac{\mathbb{E}(v^s|\mathcal {F}_n)-1}{s})=\exp \mathbb{E}(\frac{v^s-1}{s}|\mathcal {F}_n).
\end{equation}
Let $f(x)=\frac{x^s-1}{s}-\frac{x^t-1}{t}$ with $s>t>0$ and $x>0.$ Then $f(1)=0$ is the minimum value of $f$ on $(0,+\infty).$
It follows that for all $x>0$ we have $\frac{x^s-1}{s}\downarrow\log x,$ as $s\downarrow0^+.$ Using Monotone Convergence Theorem for the conditional expectation (\cite[p.5]{MR1224450}), we obtain that
\begin{eqnarray*}
\mathbb{E}\big(\frac{v^{s_0}-1}{s_0}-\frac{v^s-1}{s}|\mathcal {F}_n\big)&\uparrow&\mathbb{E}\big(\frac{v^{s_0}-1}{s_0}-\log v|\mathcal {F}_n\big), ~\text{~as~}s\downarrow0^+.
\end{eqnarray*}
Thus \begin{equation}\label{eq_further}
\lim\limits_{s\rightarrow0^+}\mathbb{E}(\frac{\omega^s-1}{s}|\mathcal {F}_n)=\mathbb{E}(\log\omega|\mathcal {F}_n).
\end{equation}
Combining \eqref{eq_left}, \eqref{eq_right} and\eqref{eq_further}, we deduce
$$\lim\limits_{s\rightarrow0^+}\mathbb{E}(\omega^s|\mathcal {F}_n)^{\frac{1}{s}}=\exp \mathbb{E}(\log\omega|\mathcal {F}_n).$$
This completes the proof of \eqref{key_lemma}
\end{proof}

To prove Theorem \ref{Thm:equa}, we use $\omega\in S$ only in $\ref{Thm:equa_R}\xRightarrow{S}\ref{Thm:equa_RH}$ and
$\ref{Thm:equa_rR}\xRightarrow{S}\ref{Thm:equa_Ap}.$

\begin{proof}[Proof of Theorem \ref{Thm:equa}] We shall follow the scheme:

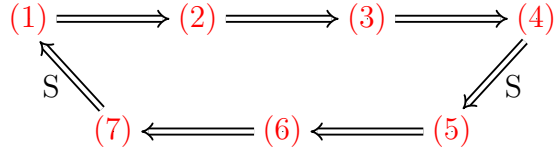
\begin{center}
\begin{figure}[H]
	\begin{tikzpicture}[scale=0.75]
	\draw (1,3) node {\ref{Thm:equa_Ap}} -- (1,3);
	\draw (4,3) node {\ref{Thm:equa_A_exp_infty}} -- (4,3);
	\draw (7,3) node {\ref{Thm:equa_wAinfty}} -- (7,3);
    \draw (10,3) node {\ref{Thm:equa_R}} -- (10,3);
	\draw (2.5,1) node {\ref{Thm:equa_rR}} -- (2.5,1);
    \draw (5.5,1) node {\ref{Thm:equa_RR}} -- (5.5,1);
    \draw (8.5,1) node {\ref{Thm:equa_RH}} -- (8.5,1);
	\draw (1.4,1.8) node {S} -- (1.4,1.8);
    \draw (9.6,1.8) node {S} -- (9.6,1.8);
    \draw[-implies,double equal sign distance, line width=0.25mm] (1.5,3) -- (3.5,3);
    \draw[-implies,double equal sign distance, line width=0.25mm] (4.5,3) -- (6.5,3);
    \draw[-implies,double equal sign distance, line width=0.25mm] (7.5,3) -- (9.5,3);
    \draw[-implies,double equal sign distance, line width=0.25mm] (9.8,2.6) -- (8.7,1.4);
    \draw[-implies,double equal sign distance, line width=0.25mm] (8,1) -- (6,1);
    \draw[-implies,double equal sign distance, line width=0.25mm] (5,1) -- (3,1);
    \draw[-implies,double equal sign distance, line width=0.25mm] (2.3,1.4) -- (1.2,2.6);
	\draw (0,0) node {} -- (0,0);
	\draw (0,3) node {} -- (0,3);
	\end{tikzpicture}
	\end{figure}
\end{center}

$\ref{Thm:equa_Ap}\Rightarrow\ref{Thm:equa_A_exp_infty}.$ Let $\omega\in A_p.$ Then for all $q\geq p,$ we have $\omega\in A_q.$ In view of Lemma \ref{key_lemma_sta} with $s=\frac{1}{q-1}$ and $v=\frac{1}{\omega},$ we obtain that
$$
\mathbb{E}(\omega^{-\frac{1}{q-1}}|\mathcal {F}_n)^{{q-1}}\downarrow\exp \mathbb{E}(\log \frac{1}{\omega}|\mathcal {F}_n),~\text{~as~}q\uparrow+\infty.
$$
Thus
$$\mathbb{E}(\omega|\mathcal {F}_n)\exp \mathbb{E}(\log\frac{1}{\omega}|\mathcal
{F}_n)\leq C,$$ which implies  $\mathbb{E}(\omega|\mathcal {F}_n)\leq C\exp \mathbb{E}(\log\omega|\mathcal
{F}_n).$

$\ref{Thm:equa_A_exp_infty}\Rightarrow\ref{Thm:equa_wAinfty}.$
Fix $n\in \mathbb{N}.$
Letting $v_n=\exp
\mathbb{E}(\log\omega|\mathcal {F}_n),$ we have that
\begin{eqnarray*}
1=\frac{1}{v_n}\exp\mathbb{E}(\log\omega|\mathcal {F}_n)
=\exp\mathbb{E}(\log\frac{\omega}{v_n}|\mathcal {F}_n).
\end{eqnarray*}
It follows that
\begin{equation}\label{equa0}\mathbb{E}(\log\frac{\omega}{v_n}|\mathcal {F}_n)=0.
\end{equation}
Using \eqref{A_exp_infty}, we obtain that
\begin{equation}\label{eqbound}\frac{1}{v_n}\mathbb{E}(\omega|\mathcal {F}_n)
\leq\frac{C}{v_n}\exp
\mathbb{E}(\log\omega|\mathcal {F}_n)
=C.\end{equation}
For some $\gamma>0$ to
be chosen later, we observe that
\begin{eqnarray*}
\{\omega\leq\gamma\omega_n\}
&=&\{\frac{\omega}{v_n}\leq
\frac{\gamma}{v_n}\mathbb{E}(\omega|\mathcal {F}_n)\}\\
&\subseteq&\{\frac{\omega}{v_n}\leq\gamma C\}\\
&\subseteq&\{\log(1+\frac{v_n}{\omega})\geq\log(1+\frac{1}{\gamma
C})\}.\end{eqnarray*}
Thus
\begin{eqnarray*}
\mathbb{E}(\chi_{\{\omega\leq\gamma\omega_n\}}|\mathcal{F}_n)
&\leq&\mathbb{E}(\chi_{\{\log(1+\frac{v_n}{\omega})\geq\log(1+\frac{1}{\gamma C})\}}|\mathcal{F}_n)\\
&\leq&\frac{1}{\log(1+\frac{1}{\gamma
                   C})}\mathbb{E}\big(\log(1+\frac{v_n}{\omega})|\mathcal
                   {F}_n\big)\\
&=&\frac{1}{\log(1+\frac{1}{\gamma
                   C})}\Big(\mathbb{E}\big(\log(1+\frac{\omega}{v_n})|\mathcal
                   {F}_n\big)-\mathbb{E}\big(\log\frac{\omega}{v_n}|\mathcal
                   {F}_n\big)\Big)\\
&=&\frac{1}{\log(1+\frac{1}{\gamma
                   C})}\mathbb{E}\big(\log(1+\frac{\omega}{v_n})|\mathcal
                   {F}_n\big),\end{eqnarray*}
where we have used \eqref{equa0}.
It follows from \eqref{eqbound} that
\begin{eqnarray*}
\frac{1}{\log(1+\frac{1}{\gamma
                   C})}\mathbb{E}\big(\log(1+\frac{\omega}{v_n})|\mathcal
                   {F}_n\big)
&\leq&\frac{1}{\log(1+\frac{1}{\gamma
                   C})}\mathbb{E}(\frac{\omega}{v_n}|\mathcal{F}_n)\\
&\leq&\frac{C}{\log(1+\frac{1}{\gamma C})}.
\end{eqnarray*}
Since
$\lim\limits_{\gamma\rightarrow0}\frac{C}{\log(1+\frac{1}{\gamma
                   C})}=0,$ we deduce that \eqref{wAinfty} holds.

$\ref{Thm:equa_wAinfty}\Rightarrow\ref{Thm:equa_R}.$ Fix $n\in \mathbb{N}.$ Suppose that $A\in \mathcal {F}$
with $\mathbb{E}_\omega(\chi_A|\mathcal {F}_n)>\beta$ for some $\beta$ to be
chosen later. Denote $A_1^c=A^c\cap\{\omega>\gamma\omega_n\}$ and
$A_2^c=A^c\cap\{\omega\leq\gamma\omega_n\}.$ Then we have
$$\mathbb{E}(\chi_{A^c}|\mathcal {F}_n)= \mathbb{E}(\chi_{A_1^c}|\mathcal {F}_n)+\mathbb{E}(\chi_{A_2^c}|\mathcal {F}_n)
\leq \frac{1}{\gamma\omega_n}\mathbb{E}(\omega\chi_{A_1^c}|\mathcal
{F}_n)+\delta.$$ Since $\mathbb{E}(\omega\chi_{A^c}|\mathcal
{F}_n)=\mathbb{E}_\omega(\chi_{A^c}|\mathcal {F}_n)\omega_n,$ it follows that
$$\mathbb{E}(\chi_{A^c}|\mathcal
{F}_n)\leq\frac{1}{\gamma}\mathbb{E}_\omega(\chi_{A^c}|\mathcal
{F}_n)+\delta<\frac{1-\beta}{\gamma}+\delta.$$ Because of
$\lim\limits_{\beta\rightarrow1}(\frac{1-\beta}{\gamma}+\delta)=\delta<1,$ it is possible to choose $\beta\in(0,1)$ such that
$\alpha=1-\frac{1-\beta}{\gamma}-\delta\in(0,1).$ Therefore, we
obtain that $$\mathbb{E}(\chi_{A^c}|\mathcal {F}_n)<1-\alpha,$$ which implies $\mathbb{E}(\chi_A|\mathcal {F}_n)>\alpha.$
Thus \ref{Thm:equa_R} is valid.

$\ref{Thm:equa_R}\xRightarrow{S}\ref{Thm:equa_RH}.$ Fix $n\in \mathbb{N}.$ Let
$\tilde{\omega}=:\frac{\omega}{\omega_n}.$ For $m\in \mathbb{N},$ denote
$\tilde{\mathcal {F}}_m=:\mathcal {F}_{n+m}.$ Since $\omega\in
S,$ we obtain a constant $C$ such that
$$\frac{1}{C}\mathbb{E}(\tilde{\omega}|\tilde{\mathcal {F}}_m)\leq
\mathbb{E}(\tilde{\omega}|\tilde{\mathcal {F}}_{m+1})\leq C
\mathbb{E}(\tilde{\omega}|\tilde{\mathcal {F}}_m),~\forall~m\in \mathbb{N}.$$ For
$k\in \mathbb{N}$, define
$\tilde{\tau}_k:=\inf\{m:\mathbb{E}(\tilde{\omega}|\tilde{\mathcal
{F}}_m)>\frac{1}{2}2^{kL}\},$ where $L\geq1$ is a large integer to
be chosen momentarily. Trivially, $\tilde{\tau}_0\equiv0$ and
$\tilde{\tau}_k\geq1,~\forall~k\geq1.$ For $k,~m\geq1,$ we have

\begin{eqnarray*}
\mathbb{E}(\chi_{\{\tilde{\tau}_{k+1}<\infty\}}\chi_{\{\tilde{\tau}_k=m\}}|\tilde{\mathcal
             {F}}_{\tilde{\tau}_k})
&\leq&\mathbb{E}\big(\frac{E(\tilde{\omega}|\mathcal
             {F}_{\tilde{\tau}_{k+1}})}{\frac{1}{2}2^{(k+1)L}}\chi_{\{\tilde{\tau}_k=m\}}|\tilde{\mathcal{F}}_m\big)\\
&=&2\mathbb{E}\big(\frac{E(\tilde{\omega}|\mathcal
             {F}_{\tilde{\tau}_{k+1}})}{2^{(k+1)L}}|\tilde{\mathcal{F}}_{\tilde{\tau}_k}\big)\chi_{\{\tilde{\tau}_k=m\}}\\
&=&\frac{2}{2^{(k+1)L}}\mathbb{E}\big(\mathbb{E}(\tilde{\omega}|\tilde{\mathcal
             {F}}_{\tilde{\tau}_{k+1}})|\tilde{\mathcal{F}}_{\tilde{\tau}_k}\big)\chi_{\{\tilde{\tau}_k=m\}}.
\end{eqnarray*} It is clear that $\tilde{\tau}_k\leq\tilde{\tau}_{k+1},$ then
\begin{eqnarray*}
\mathbb{E}(\chi_{\{\tilde{\tau}_{k+1}<\infty\}}\chi_{\{\tilde{\tau}_k=m\}}|\tilde{\mathcal
             {F}}_{\tilde{\tau}_k})
&\leq&\frac{2}{2^{(k+1)L}}\mathbb{E}(\tilde{\omega}|\tilde{\mathcal
             {F}}_{\tilde{\tau}_k})\chi_{\{\tilde{\tau}_k=m\}}\\
&=&\frac{2}{2^{(k+1)L}}\mathbb{E}(\tilde{\omega}|\tilde{\mathcal
             {F}}_m)\chi_{\{\tilde{\tau}_k=m\}}\\
&\leq&\frac{2C}{2^{(k+1)L}}\mathbb{E}(\tilde{\omega}|\tilde{\mathcal
             {F}}_{m-1})\chi_{\{\tilde{\tau}_k=m\}}\\
&\leq&\frac{2^{kL}C}{2^{(k+1)L}}\chi_{\{\tilde{\tau}_k=m\}}.\end{eqnarray*}

Here
we choose $L$ so large that $\frac{2^{kL}C}{2^{(k+1)L}}\leq \alpha,$
then$$
\mathbb{E}(\chi_{\{\tilde{\tau}_{k+1}<\infty\}}\chi_{\{\tilde{\tau}_k=m\}}|\tilde{\mathcal
             {F}}_{\tilde{\tau}_k})
\leq\alpha\chi_{\{\tilde{\tau}_k=m\}}.$$ Moreover,
\begin{eqnarray*}
\mathbb{E}(\chi_{\{\tilde{\tau}_{k+1}<\infty\}\cap{\{\tilde{\tau}_k=m\}}}|\mathcal
             {F}_{n+m})
&=&\mathbb{E}(\chi_{\{\tilde{\tau}_{k+1}<\infty\}}\chi_{\{\tilde{\tau}_k=m\}}|\mathcal
             {F}_{n+m})\\
&=&\mathbb{E}(\chi_{\{\tilde{\tau}_{k+1}<\infty\}}\chi_{\{\tilde{\tau}_k=m\}}|\tilde{\mathcal
             {F}}_m)\\
&=&\mathbb{E}(\chi_{\{\tilde{\tau}_{k+1}<\infty\}}\chi_{\{\tilde{\tau}_k=m\}}|\tilde{\mathcal
             {F}}_{\tilde{\tau}_k})\\
&\leq&\alpha\chi_{\{\tilde{\tau}_k=m\}}\leq\alpha.\end{eqnarray*}

 Combining with
$\mathbb{E}_{\tilde{\omega}}(\cdot|\mathcal
             {F}_{n+m})=\mathbb{E}_\omega(\cdot|\mathcal
             {F}_{n+m})$ and\eqref{R}, we have $$
\mathbb{E}_{\tilde{\omega}}(\chi_{\{\tilde{\tau}_{k+1}<\infty\}\cap{\{\tilde{\tau}_k=m\}}}|\mathcal
             {F}_{n+m})\leq\beta.$$
Thus
\begin{eqnarray*}
\mathbb{E}_{\tilde{\omega}}(\chi_{\{\tilde{\tau}_{k+1}<\infty\}}\chi_{{\{\tilde{\tau}_k=m\}}}|\tilde{\mathcal
             {F}}_{\tilde{\tau}_k})
&=&\mathbb{E}_{\tilde{\omega}}(\chi_{\{\tilde{\tau}_{k+1}<\infty\}\cap{\{\tilde{\tau}_k=m\}}}|\mathcal
             {F}_{n+m})\chi_{\{\tilde{\tau}_k=m\}}\\
&\leq&\beta\chi_{\{\tilde{\tau}_k=m\}}.\end{eqnarray*} Consequently,
\begin{eqnarray*}
\mathbb{E}_{\tilde{\omega}}(\chi_{\{\tilde{\tau}_{k+1}<\infty\}}|\tilde{\mathcal
             {F}}_{\tilde{\tau}_k})
&=&\mathbb{E}_{\tilde{\omega}}(\chi_{\{\tilde{\tau}_{k+1}<\infty\}}\chi_{\{\tilde{\tau}_k<\infty\}}|\tilde{\mathcal
             {F}}_{\tilde{\tau}_k})\\
&=&\sum\limits_{m=1}\limits^{\infty}\mathbb{E}_{\tilde{\omega}}
             (\chi_{\{\tilde{\tau}_{k+1}<\infty\}}\chi_{\{\tilde{\tau}_k=m\}}|\tilde{\mathcal
             {F}}_{\tilde{\tau}_k})\\
&=&\sum\limits_{m=1}\limits^{\infty}\beta\chi_{\{\tilde{\tau}_k=m\}}\leq\beta\chi_{\{\tilde{\tau}_k<\infty\}}.\end{eqnarray*}
It follows that \begin{eqnarray*}
\mathbb{E}_{\tilde{\omega}}(\chi_{\{\tilde{\tau}_{k+1}<\infty\}}|\mathcal
             {F}_n)
&=&\mathbb{E}_{\tilde{\omega}}(\mathbb{E}_{\tilde{\omega}}(\chi_{\{\tilde{\tau}_{k+1}<\infty\}}|\tilde{\mathcal
             {F}}_{\tilde{\tau}_k})|\tilde{\mathcal
             {F}}_{\tilde{\tau}_0})\\
&\leq&\beta
             \mathbb{E}_{\tilde{\omega}}(\chi_{\{\tilde{\tau}_k<\infty\}}|\tilde{\mathcal
             {F}}_{\tilde{\tau}_0})\\
&\leq&\beta^k
             \mathbb{E}_{\tilde{\omega}}(\chi_{\{\tilde{\tau}_1<\infty\}}|\tilde{\mathcal
             {F}}_{\tilde{\tau}_0})\\
&\leq&\beta^k,\end{eqnarray*} which is also valid for $k=0.$  For some positive
number $\varepsilon$ to be determined later, we have \begin{eqnarray*}
\mathbb{E}\big((\frac{\omega}{\omega_n})^{1+\varepsilon}|\mathcal
             {F}_n\big)
&=&\mathbb{E}(\tilde{\omega}^{1+\varepsilon}|\mathcal
             {F}_n)\\
&=&\mathbb{E}(\tilde{\omega}^\varepsilon\tilde{\omega}|\mathcal
             {F}_n)\\
&\leq&\mathbb{E}\big(M^*_n(\tilde{\omega})^\varepsilon\tilde{\omega}|\mathcal
             {F}_n\big),\end{eqnarray*} where $M^*_n(\cdot)=\sup\limits_{m\geq0}\mathbb{E}(\cdot|\mathcal
             {F}_{m+n})=\sup\limits_{m\geq0}\mathbb{E}(\cdot|\tilde{\mathcal
             {F}}_{m}).$
Because $\bigcap\limits_{k\in \mathbb{N}}\{\tilde{\tau}_k=\infty\}={\O},$
$\{\tilde{\tau}_0<\infty\}=\Omega$ and $\mathbb{E}_{\tilde{\omega}}\big(
             \chi_{\{\tilde{\tau}_0<\infty\}}|\mathcal
             {F}_n\big)=1,$
we obtain that \begin{eqnarray*}
\mathbb{E}\big((\frac{\omega}{\omega_n})^{1+\varepsilon}|\mathcal
             {F}_n\big)&\leq&\sum\limits_{k=0}\limits^{\infty}
             \mathbb{E}\big((M^*_n(\tilde{\omega})^\varepsilon\tilde{\omega}
             \chi_{\{\tilde{\tau}_k<\infty\}\cap\{\tilde{\tau}_{k+1}=\infty\}}|\mathcal
             {F}_n\big)\\
&\leq&\frac{1}{2^\varepsilon}\sum\limits_{k=0}\limits^{\infty}
             2^{(k+1)\varepsilon L}\mathbb{E}\big(\tilde{\omega}
             \chi_{\{\tilde{\tau}_k<\infty\}\cap\{\tilde{\tau}_{k+1}=\infty\}}|\mathcal
             {F}_n\big)\\
&=&\frac{1}{2^\varepsilon}\sum\limits_{k=0}\limits^{\infty}
             2^{(k+1)\varepsilon L}\mathbb{E}_{\tilde{\omega}}\big(
             \chi_{\{\tilde{\tau}_k<\infty\}\cap\{\tilde{\tau}_{k+1}=\infty\}}|\mathcal
             {F}_n\big)\\
&\leq&\frac{1}{2^\varepsilon}\sum\limits_{k=0}\limits^{\infty}
             2^{(k+1)\varepsilon L}\mathbb{E}_{\tilde{\omega}}\big(
             \chi_{\{\tilde{\tau}_k<\infty\}}|\mathcal
             {F}_n\big)\\
&\leq&\frac{1}{2^\varepsilon}\sum\limits_{k=0}\limits^{\infty}
             2^{(k+1)\varepsilon L}\beta^{k-1}\\
&=&\frac{2^{\varepsilon
(L-1)}}{\beta}\sum\limits_{k=0}\limits^{\infty}
             (2^{\varepsilon L}\beta)^k.\end{eqnarray*}
Choosing an $\varepsilon$ small enough, we have
$\sum\limits_{k=0}\limits^{\infty}
             (2^{\varepsilon L}\beta)^k<\infty.$ Thus, \eqref{RH} is
valid with $C=\frac{2^{\varepsilon
(L-1)}}{\beta}\sum\limits_{k=0}\limits^{\infty}
             (2^{\varepsilon L}\beta)^k$ and $q=1+\varepsilon.$

$\ref{Thm:equa_RH}\Rightarrow\ref{Thm:equa_RR}.$ For $A\in\mathcal {F},$ it is clear that
$\mathbb{E}_\omega(\chi_{A}|\mathcal
{F}_n)\omega_n=E(\omega\chi_{A}|\mathcal {F}_n).$ Applying
H\"{o}lder's inequality, we obtain that
$$\mathbb{E}_\omega(\chi_{A}|\mathcal {F}_n)\omega_n
\leq \mathbb{E}(\omega^{1+\varepsilon}|\mathcal
{F}_n)^{\frac{1}{1+\varepsilon}}\mathbb{E}(\chi_{A}|\mathcal
{F}_n)^{\frac{\varepsilon}{1+\varepsilon}},$$
where $\varepsilon=q-1>0.$
It follows from
\eqref{RH} that $$\mathbb{E}_\omega(\chi_{A}|\mathcal {F}_n)\leq
C^{\frac{1}{1+\varepsilon}}\mathbb{E}(\chi_{A}|\mathcal
{F}_n)^{\frac{\varepsilon}{1+\varepsilon}},$$ which implies
\eqref{RR}  with
$\varepsilon'=\frac{\varepsilon}{1+\varepsilon}.$

$\ref{Thm:equa_RR}\Rightarrow\ref{Thm:equa_rR}.$ For $\varepsilon'$ and $C$ in the assumption \ref{Thm:equa_RR}, we fix $\alpha'$ small enough such that
$\beta'=C\alpha'^{\varepsilon'}<1.$ For $B\in\mathcal {F}$ with
$\mathbb{E}(\chi_B|\mathcal {F}_n)\leq\alpha',$ it follows from \eqref{RR}
that$$\mathbb{E}_\omega(\chi_B|\mathcal {F}_n)\leq\beta'.$$ Thus for all $A\in \mathcal{F}$ we have
$$\mathbb{E}(\chi_A|\mathcal {F}_n)>1-\alpha'\Rightarrow
\mathbb{E}_\omega(\chi_A|\mathcal {F}_n)>1-\beta'.$$
Let $\alpha=1-\beta'$ and $\beta=1-\alpha'.$ Then
$$\mathbb{E}(\chi_A|\mathcal {F}_n)>\beta\Rightarrow
\mathbb{E}_\omega(\chi_A|\mathcal {F}_n)>\alpha,$$
which is equivalent to \ref{Thm:equa_rR}.

$\ref{Thm:equa_rR}\xRightarrow{S}\ref{Thm:equa_Ap}.$ Let $\omega_1=1/\omega.$ Then
$\omega_1$ is a weight relative to $\omega d\mu.$ Recalling the definition of $\mathbb{E}_{\omega}(\cdot|\mathcal {F}_n),$ we deduce that
$$\mathbb{E}_{\omega}(\cdot\omega_1|\mathcal
{F}_n)/\mathbb{E}_{\omega}(\omega_1|\mathcal
{F}_n)=\mathbb{E}(\cdot|\mathcal {F}_n),~\forall ~n\in \mathbb{N}.$$
It follows from \eqref{rR} that
$$\mathbb{E}_{\omega}(\chi_A|\mathcal {F}_n)\leq\alpha<1\Rightarrow
\mathbb{E}_{\omega}(\chi_A\omega_1|\mathcal
{F}_n)/\mathbb{E}_{\omega}(\omega_1|\mathcal
{F}_n)\leq\beta<1,~\forall~A\in
\mathcal{F}.$$ Combining with $\omega_1\in S_{(\omega d\mu)},$ similar to
$\ref{Thm:equa_R}\xRightarrow{S}\ref{Thm:equa_RH},$ we have $\varepsilon=q-1$ and $C$ such that
$$\mathbb{E}_{\omega}(\omega_1^{1+\varepsilon}|\mathcal {F}_n)\leq
C\mathbb{E}_{\omega}(\omega_1|\mathcal {F}_n)^{1+\varepsilon},~\forall~n\in
\mathbb{N},$$ that is, $$\omega_n\mathbb{E}(\omega^{-\varepsilon}|\mathcal
{F}_n)^{\frac{1}{\varepsilon}}\leq
C^{\frac{1}{\varepsilon}},~\forall~n\in \mathbb{N}.$$ Thus \eqref{Ap} is
valid with
$p=1+\frac{1}{\varepsilon}>1.$
\end{proof}

\section{$A_{\infty}$ Weights without Additional Assumptions}\label{without}
In this section, we study relations between different characterizations of $A_{\infty}$ weights without additional assumptions. These relations are showed in Figure \ref{figure}.

\begin{center}
  \begin{figure}[H]
	\begin{tikzpicture}[scale=0.85]
    \draw (1,5) node {$\boxed{A^{CF}_{\infty}}$} -- (1,5);
	\draw (4,-1) node {$\boxed{A^{con}_{\infty}}$} -- (4,-1);
    \draw (4,3) node {$\boxed{A^{log}_\infty}$} -- (4,3);
	\draw (4,5) node {$\boxed{\cup_{q>1}RH_q}$} -- (4,5);
    \draw (4,7) node {$\boxed{A^{\lambda}_{\infty}}$} -- (4,7);
    \draw (7,-1) node {$\boxed{A^{M}_\infty}$} -- (7,-1);
    \draw (7,1) node {$\boxed{A^{*}_\infty}$} -- (7,1);
	\draw (10,3) node {$\boxed{A^{med}_{\infty}}$} -- (10,3);
	\draw (10,5) node {$\boxed{A_{\infty}^{exp}}$} -- (10,5);	
    \draw (10,7) node {$\boxed{\cup_{p>1}A_p}$} -- (10,7);
    \draw (13,5) node {$\boxed{A_{\infty}^{SW}}$} -- (13,5);

    \draw[implies-implies,double equal sign distance, line width=0.25mm] (4.7,-1) -- (6.4,-1);
    \draw[implies-implies,double equal sign distance, line width=0.25mm] (1.7,5) -- (2.8,5);
    \draw[implies-implies,double equal sign distance, line width=0.25mm] (10.7,5) -- (12.3,5);

    \draw[-implies,double equal sign distance, line width=0.25mm] (4,4.5) -- (4,3.5);
	\draw[-implies,double equal sign distance, line width=0.25mm] (4,6.5) -- (4,5.5);
	\draw[-implies,double equal sign distance, line width=0.25mm] (4,2.5) -- (6.9,-0.5);

    \draw[-implies,double equal sign distance, line width=0.25mm] (4.7,2.5) -- (6.5,1.5);

    \draw[-implies,double equal sign distance, line width=0.25mm] (5.2,4.5) -- (6.9,1.5);

	\draw[-implies,double equal sign distance, line width=0.25mm] (8.95,6.5) -- (7.1,1.5);
	\draw[-implies,double equal sign distance, line width=0.25mm] (9.3,4.5)-- (7.5,1.5);

	\draw[-implies,double equal sign distance, line width=0.25mm] (10,4.5) -- (10,3.5);
	\draw[-implies,double equal sign distance, line width=0.25mm] (10,6.5) -- (10,5.5);
	\draw[-implies,double equal sign distance, line width=0.25mm] (10,2.5) -- (7.1,-0.5);
\draw (0,0) node {} -- (0,0);
	\draw (0,7.5) node {} -- (0,7.5);
	\end{tikzpicture}
	\caption{Relations without additional assumptions}
	\label{figure}
\end{figure}
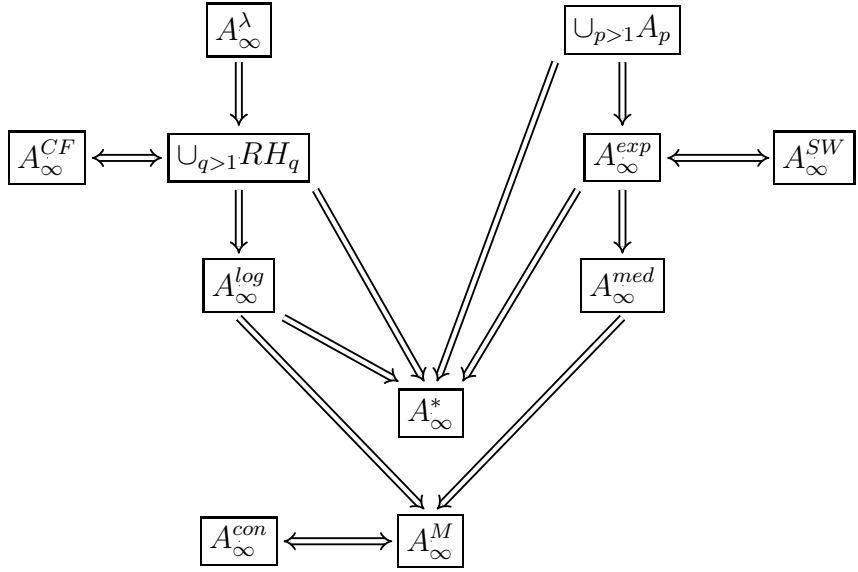
\end{center}

We first prove equivalent characterizations in Theorems \ref{thm:Reverse-test}, \ref{thm:Cond} and \ref{thm:exp-s}.

\begin{proof}[Proof of Theorem \ref{thm:Reverse-test}]
\eqref{Reverse-test1}$\Rightarrow$\eqref{Reverse-test2} This is $\ref{Thm:equa_RH}\Rightarrow\ref{Thm:equa_RR}$ in the proof of Theorem \ref{Thm:equa}. It holds without the additional assumption $\omega\in S.$

\eqref{Reverse-test2}$\Rightarrow$ \eqref{Reverse-test1} Let $B\in \mathcal {F}_n$ and let $E_{\lambda}=\{\frac{\omega}{\omega_n}>\lambda\}.$ We have
\begin{eqnarray*}
\lambda\mu(B\cap E_{\lambda})
&\leq&\int_{B\cap E_{\lambda}}\frac{\omega}{\omega_n}d\mu\\
&=&\int_{B}\frac{\omega\chi_{E_{\lambda}}}{\omega_n}d\mu\\
&=&\int_{B}\frac{\mathbb{E}(\omega\chi_{E_{\lambda}}|\mathcal
{F}_n)}{\omega_n}d\mu\\
&=&\int_{B}\mathbb{E}_{\omega}(\chi_{E_{\lambda}}|\mathcal
{F}_n)d\mu.
\end{eqnarray*}
It follows from \eqref{Reverse-test2} and H\"{o}lder's inequality that
\begin{eqnarray*}
\lambda\mu(B\cap E_{\lambda})
&\leq&C\int_{B}\mathbb{E}(\chi_{E_{\lambda}}|\mathcal{F}_n)^{\varepsilon'}d\mu\\
&\leq&C\big(\int_{B}\mathbb{E}(\chi_{E_{\lambda}}|\mathcal{F}_n)d\mu\big)^{\varepsilon'}\mu(B)^{1-\varepsilon'}\\
&=&C\mu(B\cap E_{\lambda})^{\varepsilon'}\mu(B)^{1-\varepsilon'}.
\end{eqnarray*}
Thus, we have $\lambda^{\frac{1}{1-\varepsilon'}}\mu(B\cap E_{\lambda})\leq C^{\frac{1}{1-\varepsilon'}}\mu(B).$

Let $1<q<\frac{1}{1-\varepsilon'}.$ We obtain that
\begin{eqnarray*}
\int_{B}(\frac{\omega}{\omega_n})^q\mu
&=&q\int_{0}^{+\infty}\lambda^{q-1}\mu(B\cap E_{\lambda})d\lambda\\
&=&q\int_{0}^{1}\lambda^{q-1}\mu(B\cap E_{\lambda})d\lambda
+q\int_{1}^{+\infty}\lambda^{q-1}\mu(B\cap E_{\lambda})d\lambda\\
&\leq&q(\int_{0}^{1}\lambda^{q-1}d\lambda)\mu(B)
+qC^{\frac{1}{1-\varepsilon'}}(\int_{1}^{+\infty}\lambda^{q-1-\frac{1}{1-\varepsilon'}}d\lambda)\mu(B)\\
&=&\mu(B)+\frac{C^{\frac{1}{1-\varepsilon'}}q}{\frac{1}{1-\varepsilon'}-q}\mu(B)\\
&=&(1+\frac{C^{\frac{1}{1-\varepsilon'}}q}{\frac{1}{1-\varepsilon'}-q})\mu(B).
\end{eqnarray*}
Since $B$ is arbitrary, we obtain
$$\mathbb{E}(\omega^q|\mathcal {F}_n)(\frac{1}{\omega_n})^q
\leq 1+\frac{C^{\frac{1}{1-\varepsilon'}}q}{\frac{1}{1-\varepsilon'}-q}.$$
Then $\mathbb{E}(\omega^q|\mathcal {F}_n)\leq (1+\frac{C^{\frac{1}{1-\varepsilon'}}q}{\frac{1}{1-\varepsilon'}-q})(\omega_n)^q.$
\end{proof}

\begin{proof}[Proof of Theorem \ref{thm:Cond}]
Because \eqref{thm:Cond1}$\Rightarrow$\eqref{thm:Cond2} is $\ref{Thm:equa_wAinfty}\Rightarrow\ref{Thm:equa_R}$ in the proof of Theorem \ref{Thm:equa} without the further
assumption $\omega\in S,$
it suffices to prove \eqref{thm:Cond2}$\Rightarrow$
\eqref{thm:Cond1}. Let $0<\gamma <1-\beta.$ Setting $E=\{\omega\leq\gamma\omega_n\},$ we have
\begin{eqnarray*}
\mathbb{E}_\omega(\chi_{E}|\mathcal {F}_n)&=&\frac{\mathbb{E}(\omega\chi_{E}|\mathcal {F}_n)}{\omega_n}\\
&\leq&\frac{\mathbb{E}(\gamma\omega_n\chi_{E}|\mathcal {F}_n)}{\omega_n}\\
&\leq&\gamma\mathbb{E}(\chi_{E}|\mathcal {F}_n)\\
&\leq&\gamma<1-\beta.
\end{eqnarray*}
It follows that $\mathbb{E}_\omega(\chi_{E^c}|\mathcal {F}_n)>\beta.$ In view of \eqref{thm:Cond2}, we obtain that $$\mathbb{E}(\chi_{E^c}|\mathcal {F}_n)>\alpha,$$
which implies $\mathbb{E}(\chi_{E}|\mathcal {F}_n)\leq1-\alpha.$
Thus $\eqref{thm:Cond1}$ is valid with $1-\alpha<\delta<1.$
\end{proof}

\begin{proof}[Proof of Theorem \ref{thm:exp-s}]
In view of Lemma \ref{key_lemma_sta}, we have that
\begin{equation*}
\mathbb{E}(\omega^s|\mathcal {F}_n)^{\frac{1}{s}}\downarrow\exp \mathbb{E}(\log\omega|\mathcal {F}_n),~\text{~as~}s\downarrow0^+,
\end{equation*}
which establishes the
equivalence between \eqref{thm:exp-s2} and \eqref{thm:exp-s1}.
\end{proof}

In the rest of this section, we prove Theorems \ref{thm:imp1}, \ref{thm:imp2} and \ref{thm:Wi}.
These are one-way implications.

\begin{proof}[Proof of Theorem \ref{thm:imp1}]
$\eqref{thm:lev}\Rightarrow\eqref{thm:rev}$
Let $B\in \mathcal {F}_n.$ We have
\begin{eqnarray*}
\int_B(\frac{\omega}{\omega_n})^{1+\delta}d\mu
&=&\int_B(\frac{\omega}{\omega_n})^{\delta}\frac{\omega}{\omega_n}d\mu\\
&=&\delta\int_0^{+\infty}\lambda^{\delta-1}
\frac{\omega}{\omega_n}(B\cap\{\frac{\omega}{\omega_n}>\lambda\})d\lambda\\
&=&\delta\int_0^{1}\lambda^{\delta-1}
\frac{\omega}{\omega_n}(B\cap\{\frac{\omega}{\omega_n}>\lambda\})d\lambda\\
&&+\delta\int_1^{+\infty}\lambda^{\delta-1}
\frac{\omega}{\omega_n}(B\cap\{\frac{\omega}{\omega_n}>\lambda\})d\lambda.
\end{eqnarray*}
It follows that
\begin{eqnarray*}
\delta\int_0^{1}\lambda^{\delta-1}
\frac{\omega}{\omega_n}(B\cap\{\frac{\omega}{\omega_n}>\lambda\})d\lambda
&\leq&\frac{\omega}{\omega_n}(B)\delta\int_0^{1}\lambda^{\delta-1}
d\lambda\\
&=&\mu(B),
\end{eqnarray*}
where we have used $\frac{\omega}{\omega_n}(B)
=\int_B\frac{\omega}{\omega_n}d\mu
=\int_B\frac{\omega_n}{\omega_n}d\mu=\mu(B).$

Using \eqref{thm:equ_lev}, we obtain the following estimate
\begin{eqnarray*}
&~&\delta\int_{1}^{+\infty}\lambda^{\delta-1}
\frac{\omega}{\omega_n}(B\cap\{\frac{\omega}{\omega_n}>\lambda\})d\lambda\\
&\leq&C\delta\int_1^{+\infty}\lambda^{\delta}\mu(B\cap\{\frac{\omega}{\omega_n}>\beta\lambda\})
d\lambda\\
&=&\frac{C\delta}{\beta^{1+\delta}}\int_\beta^{+\infty}\lambda^{\delta}(B\cap\{\frac{\omega}{\omega_n}>\lambda\})
d\lambda\\
&\leq&\frac{C\delta}{(1+\delta)\beta^{1+\delta}}\int_B(\frac{\omega}{\omega_n})^{1+\delta}d\mu.
\end{eqnarray*}
Because of $\lim\limits_{\delta\rightarrow0}\frac{C\delta}{(1+\delta)\beta^{1+\delta}}=0,$
we can choose $\delta$ such that $\frac{C\delta}{(1+\delta)\beta^{1+\delta}}<\frac{1}{2}.$ Then we have
$$\int_B(\frac{\omega}{\omega_n})^{1+\delta}d\mu\leq2\mu(B).$$
Since $B\in\mathcal {F}_n$ is arbitrary, it follows that $$\mathbb{E}(\omega^{1+\delta}|\mathcal {F}_n)\leq2(\omega_n)^{1+\delta}.$$

$\eqref{thm:rev} \Rightarrow\eqref{thm:log}$
Let $E_k=\{2^k<\frac{\omega}{\omega_n}\leq2^{k+1}\}$ for $k\in \mathbb{N}.$ In view of \eqref{thm:rev}, we have
\begin{eqnarray*}
  2^{kq}\mathbb{E}(\chi_{E_k}|\mathcal {F}_n)&\leq&\mathbb{E}\big((\chi_{E_k}\frac{\omega}{\omega_n})^{q}|\mathcal {F}_n\big)\\&\leq&\mathbb{E}\big((\frac{\omega}{\omega_n})^{q}|\mathcal {F}_n\big)\\&\leq& C,
\end{eqnarray*}
which implies $\mathbb{E}(\chi_{E_k}|\mathcal {F}_n)\leq C2^{-kq}.$
It follows that
\begin{eqnarray*}
\mathbb{E}_{\omega}(\log^+\frac{\omega}{\omega_n}|\mathcal {F}_n)
&=&\mathbb{E}_{\omega}(\sum\limits_{k=0}^{+\infty}\chi_{E_k}\log^+\frac{\omega}{\omega_n}|\mathcal {F}_n)\\
&=&\sum\limits_{k=0}^{+\infty}\mathbb{E}_{\omega}(\chi_{E_k}\log^+\frac{\omega}{\omega_n}|\mathcal {F}_n)\\
&=&\sum\limits_{k=0}^{+\infty}\mathbb{E}(\chi_{E_k}\frac{\omega}{\omega_n}\log^+\frac{\omega}{\omega_n}|\mathcal {F}_n)\\
&\leq&\sum\limits_{k=0}^{+\infty}2^{k+1} (k+1)\mathbb{E}(\chi_{E_k}|\mathcal {F}_n)\\
&\leq&C\sum\limits_{k=0}^{+\infty}(k+1)2^{k+1}2^{-kq},
\end{eqnarray*}
where the series $\sum\limits_{k=0}^{+\infty}(k+1)2^{k+1}2^{-kq}$ is convergent. Then we have \eqref{thm:equ_log}.

$\eqref{thm:log}\Rightarrow\eqref{thm:imp_Cond2}$
Let $\mathbb{E}(\chi_A|\mathcal {F}_n)\leq\alpha<1.$ Recall that $ab\leq a\log a-a+e^b$ where $a>1$ and $b\geq0.$
Then
\begin{eqnarray*}
\mathbb{E}_\omega(\chi_A|\mathcal {F}_n)
&=& \mathbb{E}_\omega(\chi_{A\cap\{\frac{\omega}{\omega_n}\leq1\}}|\mathcal {F}_n)+ \mathbb{E}_\omega(\chi_{A\cap\{\frac{\omega}{\omega_n}>1\}}|\mathcal {F}_n)\\
&=& \mathbb{E}(\frac{\omega}{\omega_n}\chi_{A\cap\{\frac{\omega}{\omega_n}\leq1\}}|\mathcal {F}_n)+ \mathbb{E}(\frac{\omega}{\omega_n}\chi_{A\cap\{\frac{\omega}{\omega_n}>1\}}|\mathcal {F}_n)\\
&\leq&
\mathbb{E}(\chi_A|\mathcal {F}_n)+\frac{1}{b+1}\mathbb{E}(\frac{\omega}{\omega_n}\log^+\frac{\omega}{\omega_n}+e^b\chi_A|\mathcal {F}_n)\\
&=&
\mathbb{E}(\chi_A|\mathcal {F}_n)+\frac{1}{b+1}\mathbb{E}(\frac{\omega}{\omega_n}\log^+\frac{\omega}{\omega_n}|\mathcal{F}_n)+\frac{e^b}{b+1}\mathbb{E}(\chi_A|\mathcal {F}_n)\\
&\leq&\alpha(1+\frac{e^b}{b+1})+\frac{C}{b+1}.
\end{eqnarray*}
Setting $b=2C-1,$ we can pick an $\alpha$ small enough that
$\alpha(1+\frac{e^b}{b+1})\leq\frac{1}{4}$ because of $\lim\limits_{\alpha\rightarrow0}\alpha(1+\frac{e^b}{b+1})=0.$
Thus $\mathbb{E}_\omega(\chi_A|\mathcal {F}_n)\leq\frac{3}{4}.$
\end{proof}

\begin{proof}[Proof of Theorem \ref{thm:imp2}] It suffices to prove $\eqref{thm:imp2_Log} \Rightarrow\eqref{thm:imp2_Mid}\Rightarrow\eqref{thm:imp2_Dou},$ because $\eqref{thm:imp2_Ap}\Rightarrow\eqref{thm:imp2_Log}$ is the one $\ref{Thm:equa_Ap}\Rightarrow\ref{Thm:equa_A_exp_infty}$ in Theorem \ref{Thm:equa}.

$\eqref{thm:imp2_Log} \Rightarrow\eqref{thm:imp2_Mid}$
Let $E=\{\omega>m(\omega,n)\}.$ Using H\"{o}lder's inequality for the conditional expectation,
we have
\begin{eqnarray*}
\mathbb{E}(\omega^s\chi_E|\mathcal {F}_n)
&\leq&\mathbb{E}(\omega|\mathcal {F}_n)^s\mathbb{E}(\chi_E|\mathcal {F}_n)^{1-s}\\
&\leq&2^{s-1}C^s\mathbb{E}(\omega^s|\mathcal {F}_n),
\end{eqnarray*}
where we have used Theorem \ref{thm:exp-s}.
It follows that $\mathbb{E}(\omega^s\chi_E|\mathcal {F}_n)
\leq\frac{3}{4}\mathbb{E}(\omega^s|\mathcal {F}_n)$ provided $2^{s-1}C^s<\frac{3}{4}.$
Then $\mathbb{E}(\omega^s\chi_{E^c}|\mathcal {F}_n)
\geq\frac{1}{4}\mathbb{E}(\omega^s|\mathcal {F}_n).$ Thus
\begin{eqnarray*}
\frac{1}{4}\mathbb{E}(\omega|\mathcal {F}_n)^s &\leq&\frac{1}{4}C^s\mathbb{E}(\omega^s|\mathcal {F}_n)\\ &\leq&C^s\mathbb{E}(\omega^s\chi_{E^c}|\mathcal {F}_n)\\
&\leq&C^s\mathbb{E}(m(\omega,n)^s\chi_{E^c}|\mathcal {F}_n)\\
&\leq&C^s\mathbb{E}(m(\omega,n)^s|\mathcal {F}_n)\\
&=&C^sm(\omega,n)^s,
\end{eqnarray*}
which implies $\mathbb{E}(\omega|\mathcal {F}_n)\leq4^{\frac{1}{s}}Cm(\omega,n).$

$\eqref{thm:imp2_Mid}\Rightarrow\eqref{thm:imp2_Dou}$ Let $\alpha<\frac{1}{4}$ and $\mathbb{E}(\chi_A|\mathcal {F}_n)<\frac{1}{4}.$ We claim that $$\mathbb{E}(\chi_{A^c\cap\{\omega\geq m(\omega,n)\}}|\mathcal {F}_n)\geq\frac{1}{4}.$$
Indeed, we have
\begin{eqnarray*}\mathbb{E}(\chi_{A\cup\{\omega< m(\omega,n)\}}|\mathcal {F}_n)
&\leq&\mathbb{E}(\chi_A|\mathcal {F}_n)+\mathbb{E}(\chi_{\{\omega< m(\omega,n)\}}|\mathcal {F}_n)\\
&<&\frac{1}{4}+\frac{1}{2}=\frac{3}{4}.\end{eqnarray*}
This proves that $$\mathbb{E}(\chi_{A^c\cap\{\omega\geq m(\omega,n)\}}|\mathcal {F}_n)\geq\frac{1}{4}.$$
It follows that
\begin{eqnarray*}
\frac{1}{4}\omega_n
&\leq& \frac{C}{4}m(\omega,n)\\
&\leq& Cm(\omega,n)\mathbb{E}(\chi_{A^c\cap\{\omega\geq m(\omega,n)\}}|\mathcal {F}_n)\\
&=& C\mathbb{E}(m(\omega,n)\chi_{A^c\cap\{\omega\geq m(\omega,n)\}}|\mathcal {F}_n)\\
&\leq& C\mathbb{E}(\omega\chi_{A^c\cap\{\omega\geq m(\omega,n)\}}|\mathcal {F}_n)\\
&\leq& C\mathbb{E}(\omega\chi_{A^c}|\mathcal {F}_n).
\end{eqnarray*}
Then we have $\omega_n\leq 4C\mathbb{E}(\omega\chi_{A^c}|\mathcal {F}_n)$ which implies $1\leq 4C\mathbb{E}_{\omega}(\chi_{A^c}|\mathcal {F}_n).$ Thus $\mathbb{E}_{\omega}(\chi_{A}|\mathcal {F}_n)<\beta$ with $\beta=1-\frac{1}{4C}.$
\end{proof}

Before we prove Theorem \ref{thm:Wi}, we make a couple of observations on the tailed maximal operator which are Lemmas
\ref{doob_con} and \ref{doob_imp}. Lemma \ref{doob_con} is
the conditional version of Doob's inequality which appeared in \cite[p.189]{MR1301765}.

\begin{lemma}\label{doob_con} Let $p>1.$ There exists $C>1$ such that for all $n\in \mathbb{N}$ we have
$$\mathbb{E}(M^*_n(f)^{p}|\mathcal {F}_n)\leq C\mathbb{E}(f^{p}|\mathcal {F}_n).$$
\end{lemma}

Lemma \ref{doob_imp} shows that the tailed operator has the following local property.

\begin{lemma}\label{doob_imp} Let $f\in L^1.$ For all $\lambda>0$ and $n\in\mathbb{N}$ we have
$$\mu(B\cap\{M^*_n(f)>\lambda\}\leq\frac{2}{\lambda}\int_{B\cap\{|f|>\lambda/2\}}|f|d\mu,$$
where $B\in \mathcal{F}_n.$
\end{lemma}

\begin{proof} [Proof of Lemma \ref{doob_imp}] Because of $B\in \mathcal{F}_n,$ we have
\begin{eqnarray*}
&~&\mu(B\cap\{M^*_n(f)>\lambda\}\\&=&\mu(\{M^*_n(f\chi_{B})>\lambda\})\\
&\leq&\mu(\{M^*_n(f\chi_{B\cap\{|f|>\lambda/2\}})>\lambda/2\})+\mu(\{M^*_n(f\chi_{B\cap\{|f|\leq\lambda/2\}})>\lambda/2\})\\
&=&\mu(\{M^*_n(f\chi_{B\cap\{|f|>\lambda/2\}})>\lambda/2\}).
\end{eqnarray*}

For $m\in \mathbb{N},$ denote
$\tilde{\mathcal {F}}_m=:\mathcal {F}_{n+m}$ and $\tilde{f}_m=:\mathbb{E}(f\chi_{B\cap\{|f|>\lambda/2\}}|\mathcal {F}_{n+m}).$
It follows that $(\tilde{f}_m)_{m\geq0}$ is a martingale with respect
 to $(\Omega,\mathcal{F},\mu,(\tilde{\mathcal{F}}_m)_{m\geq0}),$ which leads to
$\tilde{M}(\cdot)=M^*_n(\cdot).$ Using the weak $(1,1)$ type inequality for $\tilde{M}(\cdot),$ we have
\begin{eqnarray*}
\mu(\{M^*_n(f\chi_{B\cap\{|f|>\lambda/2\}})>\lambda/2\})
&\leq&\frac{2}{\lambda}\int_{B\cap\{|f|>\lambda/2\}}|f|d\mu,
\end{eqnarray*}
which completes the proof.
\end{proof}

Using these Lemmas, we give the proof of Theorem \ref{thm:Wi}.

\begin{proof}[Proof of Theorem \ref{thm:Wi}] In each case, we show that there exists $C>1$ such that
$$\int_{B}M^*_n(\omega)d\mu\leq C\int_{B}\omega d\mu,$$
for all $B\in \mathcal {F}_n.$ Since B is arbitrary, this proves $\mathbb{E}(M^*_n(\omega)|\mathcal {F}_n)\leq C\omega_n.$

\eqref{thm:Wi_Ap}
Because $\omega\in A_p,$ we have $\omega^{1-p^\prime}\in A_{p^\prime}$ with $1/p+1/p'=1.$ Recall that $M^*_n$ is bounded on $L^{p^\prime}(\omega^{1-p^\prime})$ (\cite[Theorem 6.6.3]{MR1224450}). Then
\begin{eqnarray*}
\int_{B}M^*_n(\omega)d\mu
&=&\int_{B}M^*_n(\omega)\omega^{-\frac{1}{p}}\omega^{\frac{1}{p}}d\mu\\
&\leq&(\int_{B}M^*_n(\omega)^{p^\prime}\omega^{-\frac{1}{p-1}}d\mu)^{\frac{1}{p^{\prime}}}
(\int_{B}\omega d\mu)^{\frac{1}{p}}\\
&\leq&C(\int_{B}\omega^{p^\prime}\omega^{-\frac{1}{p-1}}d\mu)^{\frac{1}{p^{\prime}}}
(\int_{B}\omega d\mu)^{\frac{1}{p}}\\
&=&C(\int_{B}\omega d\mu)^{\frac{1}{p^{\prime}}}
(\int_{B}\omega d\mu)^{\frac{1}{p}}\\
&\leq&C\int_{B}\omega d\mu.
\end{eqnarray*}

\eqref{thm:Wi_rev} By Jensen's inequality for the
conditional expectation, we have
\begin{eqnarray*}
\int_{B}M^*_n(\omega)d\mu
&=&\int_{B}\mathbb{E}(M^*_n(\omega)|\mathcal {F}_n)d\mu\\
&\leq&\int_{B}\mathbb{E}(M^*_n(\omega)^{p}|\mathcal {F}_n)^{\frac{1}{p}}d\mu.
\end{eqnarray*}
Because of the conditional version of Doob's inequality (Lemma \ref{doob_con}) and $\omega\in\bigcup\limits_{1<q<\infty} RH_q,$ we obtain that
\begin{eqnarray*}
\int_{B}\mathbb{E}(M^*_n(\omega)^{p}|\mathcal {F}_n)^{\frac{1}{p}}d\mu
&\leq&C\int_{B}\mathbb{E}(\omega^{p}|\mathcal {F}_n)^{\frac{1}{p}}d\mu\\
&\leq&C\int_{B}\omega d\mu.
\end{eqnarray*}
Thus
$$\int_{B}M^*_n(\omega)d\mu\leq C\int_{B}\omega d\mu.$$

\eqref{thm:Wi_log} In view of Lemma \ref{doob_imp}, we have the following estimate
\begin{eqnarray*}
\int_{B}M^*_n(\frac{\omega}{\omega_n})d\mu
&=&\int_0^{+\infty}\mu(B\cap\{M^*_n(\frac{\omega}{\omega_n})>\lambda\}
)d\lambda\\
&\leq&\int_0^{2}\mu(B)d\lambda+\int_2^{+\infty}\mu(B\cap\{M^*_n(\frac{\omega}{\omega_n})>\lambda\}
)d\lambda\\
&\leq&2\mu(B)+\int_2^{+\infty}\frac{2}{\lambda}\int_{B\cap\{ \frac{\omega}{\omega_n}>\frac{\lambda}{2}\}}\frac{\omega}{\omega_n}d\mu d\lambda\\
&=&2\mu(B)+2\int_1^{+\infty}\frac{1}{\lambda}\int_{B\cap\{\frac{\omega}{\omega_n}>\lambda\}}\frac{\omega}{\omega_n}d\mu d\lambda\\
&=&2\mu(B)+2\int_B\frac{\omega}{\omega_n}\log^+\frac{\omega}{\omega_n}d\mu\\
&=&2\mu(B)+2\int_B\mathbb{E}_{\omega}(\log^+\frac{\omega}{\omega_n}|\mathcal {F}_n)d\mu.
\end{eqnarray*}
It follows from \eqref{thm:Wi_log} that
$$\int_{B}M^*_n(\frac{\omega}{\omega_n})d\mu\leq C\mu(B).$$
Thus $\mathbb{E}\big(M^*_n(\frac{\omega}{\omega_n})|\mathcal {F}_n\big)\leq C,$ which implies
$\mathbb{E}\big(M^*_n(\omega)|\mathcal {F}_n\big)\leq C\omega_n.$

\eqref{Thm:Wi_exp} In view of Theorem \ref{thm:exp-s},
there exists $C>1$ such that for every $s\in(0,1)$ we have
$$
\mathbb{E}(\omega|\mathcal {F}_n)\leq C\mathbb{E}(\omega^s|\mathcal {F}_n)^{\frac{1}{s}}.
$$
Then $M^*_n(\omega)\leq C\cdot M^*_n(\omega^s)^{\frac{1}{s}}.$ Let $p=:\frac{1}{s}.$ Using Doob's inequality,
we have
\begin{eqnarray*}
\int_{B}M^*_n(\omega)d\mu
&\leq&C\int_{B}M^*_n(\omega^s)^{\frac{1}{s}}d\mu\\
&\leq&C(p')^p\int_{B}\omega d\mu,
\end{eqnarray*}
where $\frac{1}{p}+\frac{1}{p'}=1.$
Following from $(p')^p\downarrow e$ as $s\downarrow0,$ we obtain
$$\int_{B}M^*_n(\omega)d\mu
\leq C e \int_{B}\omega d\mu.$$
\end{proof}

\section*{Acknowledgements}The authors thank the referees for many valuable comments
and suggestions. These greatly improved the presentation of our results.

\bibliographystyle{alpha,amsplain}	

\end{document}